\documentclass[review]{elsarticle}

\usepackage{amsthm}
\usepackage{lineno}
\usepackage{amsmath,mathtools,esint}
\usepackage{amsfonts}
\usepackage{multirow}
\usepackage{multicol}
\usepackage{hhline}
\usepackage[colorlinks,bookmarksopen,bookmarksnumbered,citecolor=red,urlcolor=blue]{hyperref}
\usepackage{verbatim}
\usepackage{caption}
\usepackage{array}
\usepackage{subfigure} 
\usepackage{bigints}
\usepackage{soul}
\usepackage{amssymb}
\usepackage{wrapfig}
\usepackage{graphicx}
\usepackage[left=1in, top=1in, right=1in, bottom=1in]{geometry}
\usepackage[utf8]{inputenc}
\usepackage{lipsum}
\usepackage{amsfonts}
\usepackage{graphicx}
\usepackage{epstopdf}
\usepackage{amsopn}
\usepackage{psfrag}
\usepackage[usenames,dvipsnames]{color}
\usepackage[all]{xy}
\usepackage[normalem]{ulem}
\usepackage{MACROS/remarks,MACROS/mydef,MACROS/boldfonts,MACROS/mydef_sph,xspace}

\newtheorem{theorem}{Theorem}[section]

\newtheorem{algorithm}[theorem]{Algorithm}

\journal{ArXiv.org}

\begin{document}

\begin{frontmatter}

\title{A direction splitting scheme for Navier-Stokes-Boussinesq system in spherical shell geometries}


\author[mymainaddress]{Aziz Takhirov}
\cortext[mycorrespondingauthor]{Corresponding author}
\ead{takhirov@ualberta.ca}

\author[mymainaddress]{Roman Frolov}
\ead{frolov@ualberta.ca}

\author[mymainaddress]{Peter Minev\corref{mycorrespondingauthor}}
\ead{pminev@ualberta.ca}

\address[mymainaddress]{Department of Mathematical and Statistical Sciences, University of Alberta, Edmonton, AB, T6G 2G1, Canada}

\begin{abstract}
  This paper introduces a formally second-order direction-splitting method for
  solving the incompressible Navier-Stokes-Boussinesq system in a spherical shell region.
  The equations are solved on overset Yin-Yang grids, combined with spherical coordinate transforms. This approach allows to avoid the singularities at the poles 
  and keeps the grid size relatively uniform. The downside is that the spherical shell is subdivided into two equally sized, overlapping subdomains that requires
  the use of Schwarz-type iterations. The temporal second order accuracy is achieved via an Artificial 
  Compressibility  (AC) scheme with bootstrapping (see \cite{doi:10.1137/140975231,Guermond201792},). The spatial discretization is based on 
  second order finite differences on the Marker-And-Cell (MAC) stencil. The entire scheme is implemented in parallel using a domain decomposition
  iteration, and a direction splitting approach for the local solves.  The stability, accuracy and weak scalability of the method is verified on a manufactured solution of
  the Navier-Stokes-Boussinesq system and on the Landau solution of the Navier-Stokes equations on the sphere.
\end{abstract}


\end{frontmatter}


\section{Introduction} 
This article presents a new direction-splitting scheme for solving the incompressible Navier-Stokes-Boussinesq system:
\begin{equation}
\label{eq:eq_ns}
\begin{aligned}
\frac{\partial \bu}{\partial t}+(\bu \cdot \GRAD) \bu + \GRAD p - \Pr \LAP \bu & 
= \bg \Pr \Ra T \text{ in }~\Omega \times (0, T_{f}]\\
\nabla \cdot  \bu&=0 \text{ in }~\Omega \times (0, T_{f}] \\
\bu & = \mathbf{0}\text{ on }~ \partial \Omega \times (0, T_{f}]
\end{aligned}
\end{equation}
\begin{equation}
\label{eq:eq_heat_convec}
\begin{aligned}
\frac{\partial T}{\partial t}+(\bu \cdot \GRAD) T - \LAP T &
= 0 \text{ in }~\Omega \times (0, T_{f}] \\
T & = 0 \text{ on }~ \partial \Omega \times (0, T_{f}]
\end{aligned}
\end{equation}
in a spherical shell domain that can be defined in terms of a spherical coordinate triple $\left(r,\theta,\phi \right)$ as:
$$\Omega = \left\lbrace \left(r,\theta,\phi \right) \in [R_1, R_2] \times \left[ 0, \pi \right] \times \left[ 0, 2 \pi \right) \right\rbrace.$$
In the above, $\bg$ is the unit vector in the direction of gravity, and $\Pr, \Ra$ are the Prandtl and Rayleigh numbers, respectively. 
The system \eqref{eq:eq_ns}-\eqref{eq:eq_heat_convec} models the flow of a heat conducting fluid, 
under the assumption that the temperature-induced density variation influences significantly only the buoyancy force and the fluid
remains incompressible. It is widely applied to model the flow in the atmospheric boundary layer (\cite{MKMMKVGHJ2016}), oceanic flows (\cite{SH2006}), as well as,
if combined with an equation for the magnetic field, the flow in the Earth's dynamo (\cite{SJNF2017}).
Even though for the most part of the discussion, we assume
homogeneous Dirichlet boundary conditions on the two spherical surfaces $r=R_1, r=R_2$, the approach is applicable to Neumann and Robin boundary conditions as well.

One widely used  approach for numerical approximations of differential equations in spherical shell geometries is based on the use of a
spherical transformation that transforms the domain into a parallelepiped. 
The obvious advantage of this approach is the simple computational domain, which allows for the 
use of structured grids and the efficient schemes developed for them. Moreover, the grid can naturally follow the
geometry of the domain, without requiring too many cells, as would possibly be in the case of a Cartesian formulation.
However, the singularity of the transformation and the grid convergence near the poles
have for many years been a difficulty in the development of accurate finite difference 
and pseudo-spectral schemes. Several different treatments have been proposed for dealing with 
these problems. 
For example,  in \cite{HUANG1993254}, the pole singularity issue is avoided by replacing the equations at the poles with equations, analogous to 
boundary conditions, while in \cite{MOHSENI2000787}, a redefinition of the singular coordinates is proposed.
Other suggested approaches include applying L'Hospital's rule \cite{GRIFFIN1979352} to singular terms and
switching to Cartesian formulations around the poles \cite{doi:10.2514/6.1997-760}.
On the other hand, the grid convergence has been a more serious problem. In particular, it produces a solution with
uneven resolution, requires very small time steps for explicit or IMEX schemes since the time step size is limited by the minimum grid size, 
and causes convergence problems for iterative solvers. 
Therefore, different grid systems have been suggested in the literature that give quasi-uniform resolution and 
avoid the grid convergence problems. One such approach is the "cubed sphere" of \cite{RONCHI199693}, which is a grid that covers a 
spherical surface with six components corresponding to six faces of a cube. Even though, the resulting grid 
is quasi-uniform, it still has singularities at the corner points of the faces and it is non-orthogonal. 
Some of the other suggested unstructured grids include the isocahedral grid of \cite{Baumgardner1985} and 
non-orthogonal rhombahedral grid of \cite{doi:10.1029/2000JB900003}. 

In this study we adopt an alternative approach, proposed by \cite{doi:10.1029/2004GC000734}, employing the so-called 
Yin-Yang grids.   It starts with a decomposition of the domain into two overlapping subdomains, 
combined with  two different spherical transforms whose axes are perpendicular to each other, cf. Fig. \ref{fig:YYgrid}.
As a result, both subdomains are transformed into identical parallelepipeds that can be gridded with the same uniform grids.
This approach automatically removes the transforms singularities at the poles,
at the expense of the introduction of two subdomains, so that the two local solutions must be coupled by means of Schwarz-type iterations. 
It has been used for simulations of mantle convection \cite{TACKLEY20087}, 
core collapse supernovae \cite{refId0}, atmospherical general circulation model \cite{doi:10.1175/2010MWR3375.1}
and visualization in spherical regions \cite{OHNO20091534}.
Some advantages of Yin-Yang approach are that the metric tensors are simple, the resolution is quasi-uniform, and it requires modest 
programming effort for extending the code from a single latitude-longitude grid.
The main novelty of this paper is that the Yin-Yang domain decomposition is combined with a direction splitting time discretization that, in case of linear parabolic equations,
is unconditionally stable on grids on the spherically transformed domains.  The advection can be
included either in an IMEX fashion, or by including the linearized advection operator into the entire operator that is further split direction-wise.
The resulting splitting scheme is conditionally stable, since the direction-wise operators are not positive, but our numerical experience
demonstrated that the second approach yields an algorithm that has better stability performance.  This is why the rest of the paper 
concerns only this type of schemes.
To our knowledge, the stability of the direction splitting approach has not been rigorously studied in the context a spherical coordinate system.
Therefore, we prove below that it is unconditionally stable in case of a scalar heat equation, in a simply shaped domain (in terms of
spherical coordinates).  The case of the full Navier-Stokes-Boussinesq system is more involved and we do not provide
a rigorous proof here.  However, our numerical experience shows that the a direction splitting is still unconditionally stable if the
advection terms are omitted and if the velocity-pressure decoupling is done via the AC method proposed in \cite{doi:10.1137/140975231}.

The rest of the paper is organized as follows. In the next section, we briefly recall the definition of the Yin-Yang domain decompostion. 
In Section $3$ we present the numerical scheme for the advection-diffusion and Navier-Stokes equations on each of the subdomains.
In Section $4$, we discuss the implementation details, and in Section $5$ we present the numerical experiments.
\section{Spatial discretization and the Yin-Yang grid}
In this section, we briefly recall the definition of the composite Yin-Yang grid following \cite{doi:10.1029/2004GC000734}.
The grid consists of two identical overlapping latitude-longitude grids whose axes are perpendicular to each other. 
The Yin grid is based on a spherical transformation 
\begin{align*}
\begin{cases}
 x &= r \sin \theta \cos \phi \\
 y &= r \sin \theta \sin \phi \\
 z &= r \cos \theta ,
 \end{cases}
\end{align*}
and covers the region 
$$ \Omega_1: = \left\lbrace \left(r,\theta,\phi \right) \in [R_1, R_2] \times \left[ \frac{\pi}{4}-\varepsilon, \frac{7\pi}{4}+\varepsilon \right] \times 
\left[ \frac{\pi}{4}-\varepsilon, \frac{3\pi}{4}+\varepsilon \right]  \right\rbrace, $$
where $\varepsilon \ll 1$ is a parameter determining the overlap. The Yang grid is obtained via another spherical transformation:
\begin{align*}
\begin{cases}
 x &= - r \sin \ttheta \cos \tphi \\
 y &= r \cos \ttheta \\
 z &= r \sin \ttheta \sin \tphi,
 \end{cases}
\end{align*}
such that its axes is perpendicular to the axes of the Yin transform, and covers the region
$$ \Omega_2: = \left\lbrace \left(r,\ttheta,\tphi \right) \in [R_1, R_2] \times \left[ \frac{\pi}{4}-\varepsilon, \frac{7\pi}{4}+\varepsilon \right] \times 
\left[ \frac{\pi}{4}-\varepsilon, \frac{3\pi}{4}+\varepsilon \right]  \right\rbrace. $$  The choice of the second axes should be such that the
Yang grid fully covers the gap of the Yin one, and the overlapping subregions are of the same size (see Fig. \ref{fig:YYgrid},).  Otherwise, it is identical to the Yin grid modulo two rotations. The resulting Yin-Yang grids are quasiuniform,
the coordinate transformations from $(r, \theta, \phi)$ to $(r, \ttheta, \tphi)$  and its inverse, as well as the metric tensors on both grids are identical.
As a consequence,  the methods and codes developed for the standard latitude-longitude grid can be applied 
to both grids. 
\begin{figure}[htbp]
\centering
\begin{minipage}[b]{0.49\linewidth}
\centering
\includegraphics[width=\textwidth]{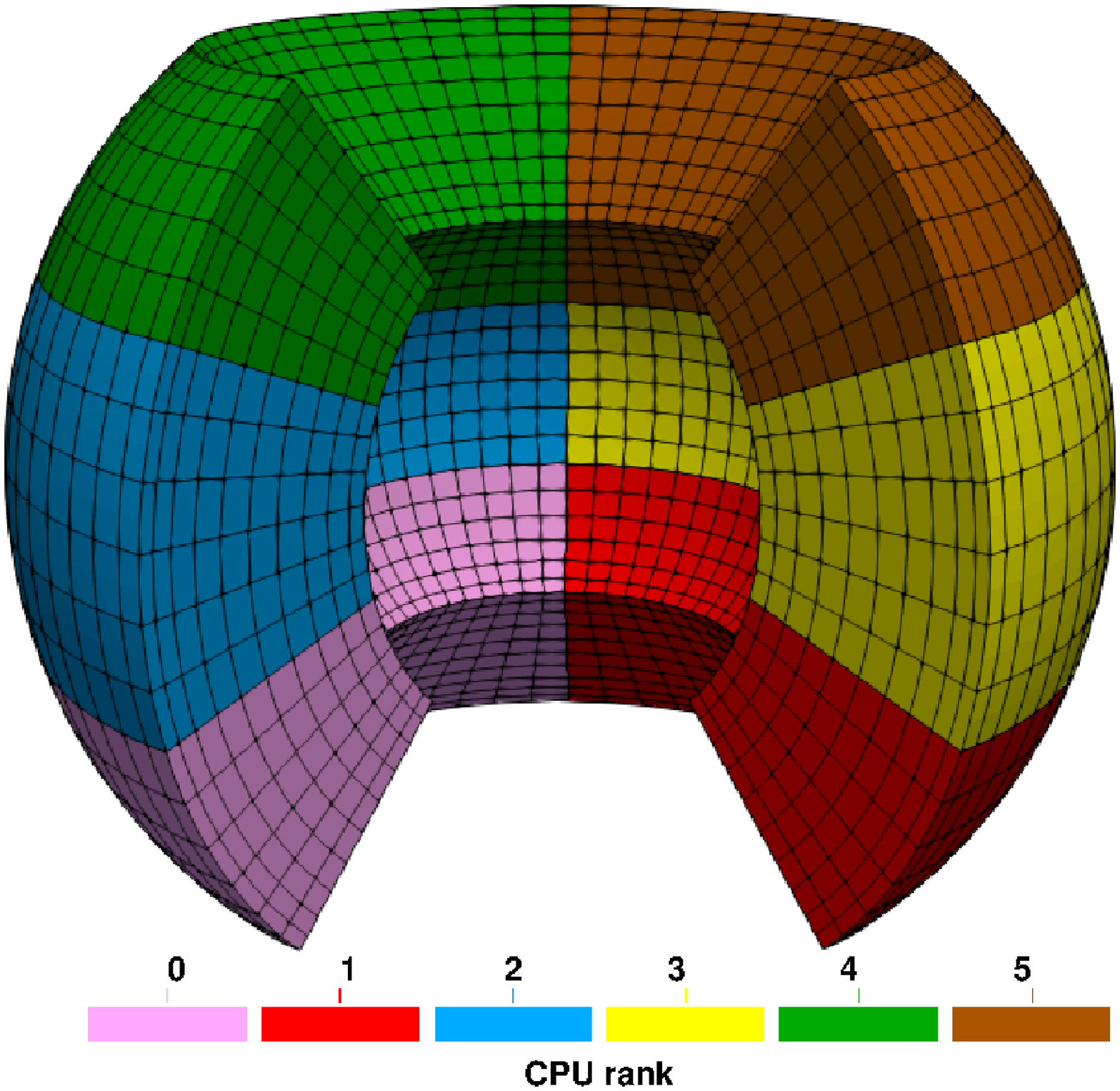}
\end{minipage}
\hspace{0.1cm}
\begin{minipage}[b]{0.49\linewidth}
\centering
\includegraphics[width=\textwidth]{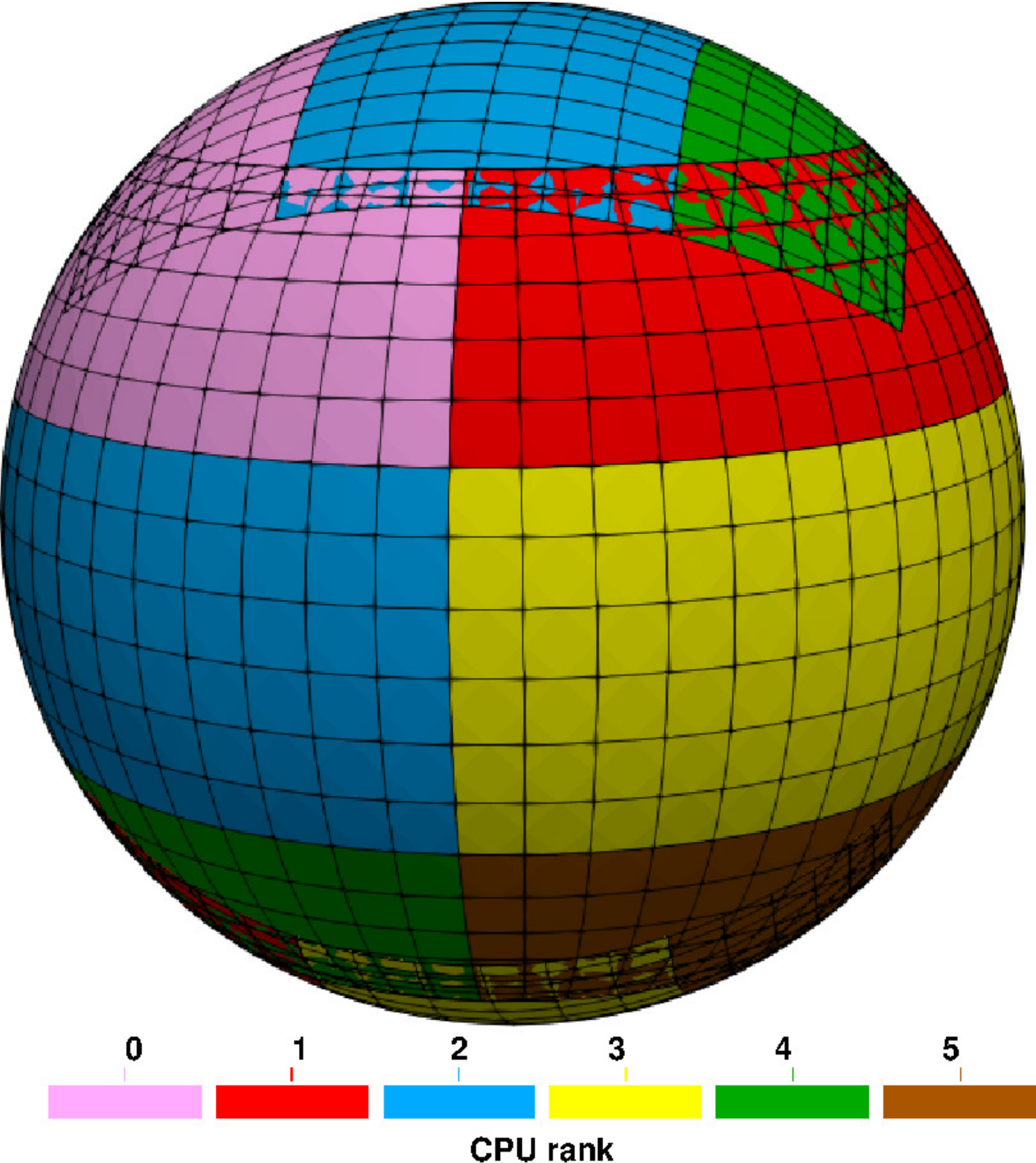}
\end{minipage}
\caption{Yang (left) and Yin-Yang (right) grids.  Each subgrid is further decomposed into blocks for a parallel implementation corresponding to a CPU distribution $1\times3\times2$.}
\label{fig:YYgrid}
\end{figure}
\section{Direction-splitting time discretization}
\subsection{Preliminaries}
In the sequel of the paper we will frequently make use of the following notations.  For a time sequence $w^k, k=1, 2,\dots$
we denote the average between two time levels as $w^{k+1/2}=(w^{k+1}+w^k)/2$, and the explicit extrapolation to level $k+1/2$
by $w^{*,k+1/2}=3w^{k}/2-w^{k-1}/2$.  For two regular enough functions $u, v$ defined in the spherical shell we denote their
 weighted $L^2$ inner product as:
$(u,v)_{\omega}: = \int \limits_{R_1}^{R_2} \int  \limits_{0}^{\pi} \int  \limits_{0}^{2 \pi}  u v \omega d r d \theta d \phi, $
where $\omega$ denotes a non-negative weight.  In most cases the weight is given by the weight of the spherical transform $\omega = r^2 \sin \theta$, 
however, in some of the estimates given below, the weight will be appropriately modified.
The corresponding norm is given by $\| u\|_{\omega}^2 = (u,u)_{\omega}$.

\subsection{Direction splitting of the advection-diffusion equation}
Since the PDEs are identical in both domains, it is sufficient to develop the numerical scheme for the Yin domain. 
Then the Schwarz domain decomposition method can be used to iterate between the subdomains.

We first present Douglas \cite{Douglas1962} type direction splitting scheme for the heat equation. 
Consider
\begin{equation}
\label{eq:eq_temp}
\begin{aligned}
\partial_t T - \kappa \LAP T &= 0  \text{ in }~\Omega_1 \times (0, T_{f}],\\
T &=0 \text{ on }~ \partial \Omega_1 \times (0, T_{f}],
\end{aligned}
\end{equation} 
where the  Laplacian in spherical coordinates is given by
$$ \Delta  = \Drr + \Dthetatheta + \Dphiphi, \Drr: = \frac{1}{r^2}\dr \left(r^2 \dr \right), 
\Dthetatheta: = \frac{1}{r^2 \sin \theta} \partial_\theta\left( \sin \theta \partial_\theta \right), 
\text{ and } \Dphiphi: = \frac{\dphiphi}{r^2 \sin^2 \theta}.$$
The Douglas direction splitting scheme for this equation can be summarized in the following factorized form:
\begin{align}
\left[ \mathrm{I} - \frac{\dt}{2} \Drr \right] \left[ \mathrm{I} - \frac{\dt}{2} \Dthetatheta \right]
\left[ \mathrm{I} - \frac{\dt}{2} \Dphiphi \right] \frac{\delta T^{n+1}}{\dt} = 
 \Delta T^{n} ,
\label{eq:HeatDouglas0}
\end{align}
where $ \delta T^{n+1} : = T^{n+1} - T^n$ denotes the first time difference of the time sequence $T^k$, $\dt$ is the time step, and $ \mathrm{I}$
is the identity operator.
We first notice that this splitting can be considered as an Euler explicit scheme whose time difference operator is
multiplied by $\left[ \mathrm{I} - \frac{\dt}{2} \Drr \right] \left[ \mathrm{I} - \frac{\dt}{2} \Dthetatheta \right]
\left[ \mathrm{I} - \frac{\dt}{2} \Dphiphi \right] $, that is a consistent perturbation of $ \mathrm{I}$ and stabilizes the scheme.
If the spatial derivative operators are positive and commute with respect to some inner product, the stability of this scheme is not hard to establish.

Unfortunately, $\Drr, \Dthetatheta,$ and $\Dphiphi$ do not commute with respect to the weighted product $(.,.)_\omega$, 
and their positivity is far from being clear.  The main obstacle to commutativity of the
one-dimensional operators comes from the non-constant terms in the denominators of $\Dthetatheta$ and $\Dphiphi$. 
Therefore, the scheme \eqref{eq:HeatDouglas0} should be modified as follows. We first introduce the modified spatial operators:
\begin{equation*}
\DHthetatheta : = \frac{1}{R_1^2 \sin \theta} \partial_\theta\left( \sin \theta \partial_\theta \right), \DHphiphi := \frac{\dphiphi}{R_1^2 \sin^2 \theta_1}, 
\ \text{ and } \hat{\LAP} := \Drr + \DHthetatheta + \DHphiphi,
\end{equation*}
where $\displaystyle \theta_1 = \frac{\pi}{4}-\varepsilon $.
Then it is easy to check that $\Drr, \DHthetatheta$ and $\DHphiphi$, supplied with zero Dirichlet boundary conditions, commute. Moreover,
\begin{equation}
-\left( \DHthetatheta T,T \right)_{\omega} \ge 0, \quad
 - \left( \DHphiphi T,T \right)_{\omega}  \ge 0 \label{eq:HeatStab2}
\end{equation}
and
\begin{equation}
- \left( \left[\DHthetatheta - \Dthetatheta \right]T,T \right)_{\omega}  \ge 0 \text{ and } 
 - \left( \left[\DHphiphi - \Dphiphi \right]T,T \right)_{\omega}  \ge 0. 
 \label{eq:HeatStab1}
\end{equation}
These inequalities immediately yield that:
\begin{equation}
\left( -\hat{\LAP} T, T \right)_{\omega}  \ge \left( - \LAP T, T \right)_{\omega} .
 \end{equation}
In order to obtain an unconditionally stable second order scheme, we start from the second order Adams-Bashforth scheme:
\begin{equation*}
 \frac{\delta T^{n+1}}{\dt} = 
 \Delta T^{*,n+1/2},
 \label{eq:AB2}
\end{equation*}
 and stabilize it by multiplying the time difference 
in the left hand side by $\left[ \mathrm{I} - \frac{\dt}{2} \Drr \right] \left[ \mathrm{I} - \frac{\dt}{2} \DHthetatheta \right] \left[ \mathrm{I} - \frac{\dt}{2} \DHphiphi \right] .$
Since this perturbation is only first order consistent with the identity operator, we
subtract from the right hand side the first order perturbation term, taken at the previous time level.  The resulting splitting scheme reads:
\begin{align}
\left[ \mathrm{I} - \frac{\dt}{2} \Drr \right] \left[ \mathrm{I} - \frac{\dt}{2} \DHthetatheta \right]
\left[ \mathrm{I} - \frac{\dt}{2} \DHphiphi \right] \frac{\delta T^{n+1}}{\dt} = 
 \Delta T^{*,n+1/2} -  \frac{1}{2} \hat{\Delta} \delta T^n .
\label{eq:HeatDouglas}
\end{align}
Note that, assuming enough regularity of the exact solution in space and time, this is a second-order perturbation of the second-order explicit Adams-Bashforth scheme \eqref{eq:AB2},
the perturbation being given by:
$$
\frac{\dt^2}{2} \hat{\Delta} \frac{\delta^2 T^{n+1}}{\dt^2} +  \left[ \frac{\dt^2}{4} (\Drr \DHthetatheta + \Drr \DHphiphi + \DHthetatheta \DHphiphi)
-  \frac{\dt^3}{8}  \Drr \DHthetatheta  \DHphiphi \right]  \frac{\delta T^{n+1}}{\dt} .
$$
We have the following stability result for the scheme \eqref{eq:HeatDouglas}.
\begin{theorem}
Assuming enough regularity of the exact solution $T$ of the semi-discrete scheme \eqref{eq:HeatDouglas}, it is unconditionally stable; more precisely, it satisfies the following estimate:
\begin{align}
\tau \sum \limits_{n=1}^{N-1} \frac{\| T^{n+1} - T^n \|^2_{\omega}}{\tau^2} &+ 
\frac{1}{2} \|\nabla T^{N} \|^2_{\omega} + \frac{1}{4} \left( \|\partial_\theta \left( T^N - T^{N-1} \right) \|^2_{\omega_1} + 
\|\partial_\phi \left( T^N - T^{N-1} \right) \|^2_{\omega_2} \right) \label{eq:Stability} \\
& \le \frac{1}{2} \|\nabla T^{1} \|^2_{\omega} + \frac{1}{4} \left( \|\partial_\theta \left( T^{1} - T^{0} \right) \|^2_{\omega_1} + \|\partial_\phi \left( T^{1} - T^{0} \right) \|^2_{\omega_2} \right), \notag
\end{align}
where $\omega_1 = \left( 1- \frac{r^2}{R^2_1} \right) \sin \theta \ge 0$ and 
$\omega_2 = \left( \frac{r^2}{R^2_1} - 1 \right) \frac{\sin \theta}{\sin^2 \theta_1} \ge 0$.
\end{theorem}
\begin{proof}
Expanding the left hand side of \eqref{eq:HeatDouglas} we get:
\begin{align}
\left[ \mathrm{I} - \frac{\dt}{2} \hat{\Delta} \right. & + \left. \frac{\dt^2}{4} \left(\Drr \DHthetatheta + 
\Drr \DHphiphi + \DHthetatheta \DHphiphi \right) 
 - \frac{\dt^3}{8}\Drr \DHthetatheta \DHphiphi \right] \frac{\delta T^{n+1}}{\dt} 
 = \Delta T^{*,n+1/2} - \frac{1}{2} \hat{\Delta} \delta T^n.
\label{eq:StabProof1}
\end{align}
Rearranging all the $\Delta$ and $\hat{\Delta}$ terms, we obtain
\begin{align}
\frac{\delta T^{n+1}}{\dt} & - \frac{1}{2} \left[\Delta - \hat{\Delta} \right] \left(T^{n+1} - 2T^n + T^{n-1} \right)
- \LAP T^{n+1/2} \notag \\
 & + \left[ \frac{\dt^2 }{4} \left(\Drr \DHthetatheta + 
\Drr \DHphiphi + \DHthetatheta \DHphiphi \right) 
 - \frac{\dt^3 }{8}\Drr \DHthetatheta \DHphiphi \right]\frac{\delta T^{n+1}}{\dt} = 0.
 \label{eq:StabProof2}
\end{align}
Next we multiply \eqref{eq:StabProof2} by $v = \delta T^{n+1}$ and integrate by parts. Then the second term gives
\begin{align}
& - \frac{1}{2} \left( \left[\Delta - \hat{\Delta} \right] \left(T^{n+1} - 2T^n + T^{n-1} \right), T^{n+1}-T^{n} \right)_{\omega} \notag \\
& = \frac{1}{4} \left[ \| \partial_\theta \left( T^{n+1} - T^n \right)\|^2_{\omega_1} - 
\| \partial_\theta \left(  T^{n} - T^{n-1} \right) \|^2_{\omega_1} + \| \partial_\theta \left( T^{n+1} - 2T^n + T^{n-1} \right)\|^2_{\omega_1} \right]
 \label{eq:StabProof3} \\
& + \frac{1}{4} \left[ \| \dphi \left( T^{n+1} - T^n \right)\|^2_{\omega_2} - 
\| \dphi \left( T^{n} - T^{n-1} \right) \|^2_{\omega_2} + \| \dphi \left( T^{n+1} - 2T^n + T^{n-1} \right)\|^2_{\omega_2} \right]. \notag
\end{align}
The third term is 
\begin{align}
- \left( \Delta T^{n+1/2}, T^{n+1}-T^{n} \right)_{\omega}
= \frac{1}{2} \left( \| \nabla T^{n+1} \|_{\omega}^2 - \|\nabla T^n \|_{\omega}^2 \right).
 \label{eq:StabProof4}
\end{align}
The remaining terms are all dissipative:
\begin{align}
 \left( \Drr \DHthetatheta \delta T^{n+1}, \delta T^{n+1} \right)_{\omega}^2
=\int \limits_{\Omega} \frac{r^2 \sin \theta}{R^2_1} | \partial_{r \theta} \delta T^{n+1} |^2,
 \label{eq:StabProof5}
\end{align}
\begin{align}
 \left( \Drr \DHphiphi \delta T^{n+1}, \delta T^{n+1} \right)_{\omega}
= \int  \limits_{\Omega}  \frac{r^2 \sin \theta}{R^2_1 \sin^2 \theta_1} | \partial_{r \phi} \delta T^{n+1} |^2,
 \label{eq:StabProof6}
\end{align}
\begin{align}
 \left( \DHthetatheta \DHphiphi \delta T^{n+1}, \delta T^{n+1} \right)_{\omega}
= \int \limits_{\Omega}  \frac{r^2 \sin \theta}{R^4_1 \sin^2 \theta_1} | \partial_{\theta \phi} \delta T^{n+1} |^2,
 \label{eq:StabProof7}
\end{align}
and
\begin{align}
- \left( \Drr \DHthetatheta \DHphiphi \delta T^{n+1}, \delta T^{n+1} \right)_{\omega}
= \int  \limits_{\Omega}  \frac{r^2 \sin \theta}{R^4_1 \sin^2 \theta_1} | \partial_{r\theta \phi} \delta T^{n+1} |^2.
 \label{eq:StabProof8}
\end{align}
Substituting \eqref{eq:StabProof3}-\eqref{eq:StabProof8} into \eqref{eq:StabProof2}, and summing for $n=1, \dots,N-1$  completes the proof.
\end{proof}
The factorized scheme for the advection-diffusion equation \eqref{eq:eq_heat_convec} is obtained in a similar fashion and takes the following form:
\begin{align}
\left[ \mathrm{I} - \frac{\dt}{2} \left( \Drr - u_r^{n+1/2} \dr \right) \right]
\left[ \mathrm{I} \right. & - \left. \frac{\dt}{2} \left( \DHthetatheta - u_\theta^{n+1/2} \frac{\dtheta}{r} \right) \right]
\left[ \mathrm{I} - \frac{\dt}{2} \left( \DHphiphi - u_\phi^{n+1/2} \frac{\dphi}{r \sin \theta} \right) \right] \frac{\delta T^{n+1}}{\dt} \notag \\
& =  \Delta T^{*,n+1/2} - \frac{1}{2} \hat{\Delta} \delta T^n + \bu^{n+1/2} \cdot \nabla T^n.
\label{eq:Convect_Douglas}
\end{align}
\subsection{Direction-splitting discretization of the Navier-Stokes system}
Now we present the direction splitting scheme for the Navier-Stokes equations \eqref{eq:eq_ns}. 
Our numerical scheme is based on the AC regularization:
\begin{equation}
\begin{aligned}
\partial_t \bu_1 +(\bu_1 \cdot \GRAD) \bu_1 + \GRAD p_1 - \frac{1}{\Reynolds} \LAP \bu_1  &
= \mathbf{0} \\
\chi \dt \partial_t p_1 + \nabla \cdot \bu_1 &=0,
\end{aligned}
\label{eq:ac1}
\end{equation}
where $\chi = \mathcal{O}\left(1 \right)$ is an artificial compressibility regularization parameter, and $\Reynolds$ is the Reynolds number. 
It is well-known that the resulting approximation  $\left( \bu_1, p_1 \right)$ is first-order accurate in time (see \cite{Shen_1995}). 
A second order scheme can be constructed using the bootstrapping approach of \cite{doi:10.1137/140975231,Guermond201792}, which 
requires additionally to solve the system:
\begin{equation}
\begin{aligned}
\partial_t \bu_2 + (\bu_2 \cdot \GRAD) \bu_2 + \GRAD p_2 - \frac{1}{\Reynolds}\LAP \bu_2 &
= \mathbf{0}  \\
\chi \dt \partial_t \left( p_2 -  p_1 \right) + \nabla \cdot  \bu_2&=0,
\end{aligned}
\label{eq:ac2}
\end{equation}
$p_1$ being given by \eqref{eq:ac1}.
In the following, for the sake of brevity, we will only discuss the direction splitting implementation of the first order approximation \eqref{eq:ac1}. The higher order correction for $\bu_2, p_2$ is solved identically.
First, consider the standard semi-implicit Crank-Nicholson approximation of the system for $\left( \bu_1, p_1 \right)$:
\begin{align*}
\frac{\bu^{n+1}_1-\bu^{n}_1}{\dt} + \bu^{*,n+1/2}_2 \cdot \GRAD \bu^{n+1/2}_1 + \GRAD p_1^{n+1/2} - \frac{1}{\Reynolds} \LAP \bu_1^{n+1/2}&
= \mathbf{0} \\
\chi \left( p_1^{n+1}-p^n_1 \right) + \nabla \cdot \bu_1^{n+1/2} &=0
\end{align*}
Note that we use the second order velocity $\bu_2$ as advecting velocity, which allows us to assemble a single linear system for 
both systems. We can rewrite the momentum equation by eliminating $p^{n+1}_1$ from the first equation:
\begin{align*}
\frac{\bu^{n+1}_1-\bu^{n}_1}{\dt} + \bu^{*,n+1/2}_2 \cdot \GRAD \bu^{n+1/2}_1 + \GRAD p_1^n - \frac{1}{\Reynolds} \LAP \bu_1^{n+1/2}& - \frac{1}{2 \chi}\GRAD \DIV \bu_1^{n+1/2} 
= \mathbf{0} \\
p_1^{n+1} = p^n_1 & - \frac{1}{\chi}\nabla \cdot \bu_1^{n+1/2}.
\end{align*}
In order to produce a factorized scheme for each velocity component, the $\GRAD \DIV$ operator must be also split somehow, and we use the 
Gauss-Seidel type splitting of the $\GRAD \DIV$ operator, which was 
originally proposed in \cite{Guermond201792} in the  Cartesian case:
$$
\GRAD \DIV \bu^{n+1/2} \simeq 
\begin{pmatrix} 
\partial_r \biggl( \frac{\dr \left(r^2 u^{n+1/2}_r \right)}{r^2} + \frac{\dtheta \left( \sin \theta u_\theta^{*,n+1/2} \right)}{r \sin \theta} 
 + \frac{\dphi u_\phi^{*,n+1/2}}{r \sin \theta}  \biggr) \\
\\
\frac{\partial_\theta}{r} \biggl( \frac{\dr \left(r^2 u_r^{n+1/2} \right)}{r^2} + \frac{\dtheta \left( \sin \theta u^{n+1/2}_\theta \right)}{r \sin \theta} 
 + \frac{\dphi u_\phi^{*,n+1/2}}{r \sin \theta}  \biggr) \\
\\
\frac{\partial_\phi}{r \sin \theta} \biggl( \frac{\dr \left(r^2 u_r^{n+1/2} \right)}{r^2} + \frac{\dtheta \left( \sin \theta u_\theta^{n+1/2} \right)}{r \sin \theta} 
 + \frac{\dphi u^{n+1/2}_\phi}{r \sin \theta}  \biggr)
\end{pmatrix} :=
\begin{pmatrix}
D_{11} + D_{12} + D_{13} \\
D_{21} + D_{22} + D_{23} \\
D_{31} + D_{32} + D_{33} 
\end{pmatrix} \bu^{n+1/2}
$$
\subsubsection{Equation for the $r$-component of the velocity} 
Using the mass conservation equation $\nabla \cdot \bu = 0$, it is possible to write the first component of the system as follows:
\begin{equation*}
\partial_t u_r + \bu \cdot \nabla u_r - \frac{\LAP u_r}{\Reynolds} + \dr p 
+ \frac{1}{\Reynolds}\frac{2 u_r}{ r^2} - \frac{1}{\Reynolds} \frac{2}{r^3}\partial_r \left(u_r r^2 \right) - \frac{u^2_\theta + u^2_\phi}{r} = 0, 
\end{equation*}
where $\bu \cdot \GRAD v = u_r \dr v + u_\theta \frac{\dtheta v}{r} + u_\phi \frac{\dphi v}{r \sin \theta}$ is the advection operator.
Let $\Lrr, \Lrtheta$ and $\Lrphi$ be the differential operators that act in each space direction:  
$$\Lrr u  = \frac{1}{\Reynolds} \left( \Drr u  - \frac{2 u}{r^2} + \frac{2 \partial_r \left(r^2 u\right)}{r^3} \right) + D_{11}u - 
u^{*,n+1/2}_{2,r} \cdot \dr u_r, \Lrtheta u  = \left(\frac{\DHthetatheta }{\Reynolds} - u^{*,n+1/2}_{2,\theta} \cdot \frac{\dtheta }{r} \right)u,$$
$$\Lrphi u  = \left( \frac{\DHphiphi }{\Reynolds} - u^{*,n+1/2}_{2,\phi} \cdot \frac{\dphi }{r \sin \theta} \right)u \text{ and }
L_r = \Lrr + \Lrtheta + \Lrphi$$
The factorized scheme for the $r$-component takes the following form:
\begin{align}
\left[ \mathrm{I} - \frac{\dt}{2} \Lrtheta \right] 
\left[ \mathrm{I} - \frac{\dt}{2} \Lrphi \right]
\left[ \mathrm{I} - \frac{\dt}{2} \Lrr \right] \frac{u_{1,r}^{n+1}-u_{1,r}^{n}}{\dt} 
& = L_r u_{1,r}^{*,n+1/2} + \hat{\LAP} u^{n-1/2}_{1,r} - \dr p_1^{n} 
+ \frac{D_{12} u_{1,\theta}^{*,n+1/2} + D_{13}u_{1,\phi}^{*,n+1/2}}{2 \chi} \notag \\
& + \frac{ \left( u^{*,n+1/2}_\theta \right)^2 + \left( u^{*,n+1/2}_\phi \right)^2}{r}. \label{eq:r_comp}
\end{align}
\subsubsection{Equation for the $\theta$--component of the velocity}
Again using $\nabla \cdot \bu = 0$, the $\theta$-component of the momentum equation can be expressed as:
 \begin{align*}
\partial_t u_\theta + \bu \cdot \nabla u_\theta - \frac{\LAP u_\theta}{\mathrm{Re}} + \frac{\dtheta p}{r} 
+ \frac{1}{\Reynolds}\frac{u_\theta}{r^2 \sin^2 \theta} - \frac{2 \cos \theta}{\Reynolds}\frac{\dtheta \left(u_\theta \sin \theta \right)}{r^2 \sin^2 \theta} 
- \frac{2}{\Reynolds} \frac{\dtheta u_r}{r^2}  & - \frac{2 \cos \theta}{\Reynolds} \frac{\dr \left(u_r r^2 \right)}{r^3 \sin \theta} \notag \\
& + \frac{u_r u_\theta - u_\phi ^2 \cot \theta}{r}= 0. 
\end{align*}
Let $\Lthetar, \Lthetatheta$ and $\Lthetaphi$ be defined as follows:
$$\Lthetar u  = \left(\frac{\Drr }{\mathrm{Re}} - u^{*,n+1/2}_{2,r} \cdot \dr \right) u,   
\Lthetaphi u  = \left( \frac{\DHphiphi }{\mathrm{Re}} - u^{*,n+1/2}_{2,\phi} \cdot \frac{\dphi }{r \sin \theta} \right) u, $$
$$ \Lthetatheta u = \frac{1}{\Reynolds} \left( \DHthetatheta u - \frac{u}{r^2 \sin^2 \theta} +  
\frac{2 \cos \theta}{\sin \theta}\partial_\theta \left(u \sin \theta \right) \right) + \frac{u \cdot u^{*,n+1/2}_{2,\phi} \cot \theta}{r}
+ u^{*,n+1/2}_{2,\theta} \cdot \frac{\dtheta u}{r} + \frac{D_{22} u}{2 \chi}, $$
$$ \text{ and } L_\theta = \Lthetar + \Lthetatheta + \Lthetaphi$$
The factorized scheme for the $\theta$-component takes the following form:
\begin{align}
\left[ \mathrm{I} - \frac{\dt}{2} \Lthetaphi \right]
\left[ \mathrm{I} - \frac{\dt}{2} \Lthetar \right]
\left[ \mathrm{I} - \frac{\dt}{2} \Lthetatheta \right] \frac{u_{1,\theta}^{n+1}-u_{1,\theta}^{n}}{\dt} 
& = L_\theta u_{1,\theta}^{*,n+1/2} + \hat{\LAP} u^{n-1/2}_{1,\theta} - \frac{\dtheta p_1^{n}}{r} + \frac{D_{21} u_{1,r}^{n+1/2} + D_{23} u_{1,\phi}^{*,n+1/2}}{2 \chi} \notag \\
& + \frac{1}{\Reynolds} \left( \frac{2 }{r^2} \dtheta u_{1,r}^{n+1/2} + \frac{2 \cos \theta}{r^3 \sin \theta} \dr \left(u_{1,r}^{n+1/2} r^2 \right) \right) \label{eq:theta_comp} \\
& - \frac{u_r^{*,n+1/2} \cdot u^{*,n+1/2}_\phi}{r}. \notag
\end{align}
\subsubsection{Equation for the $\phi$--component of the velocity} 
The $\phi$-component of the momentum equation is given by:
 \begin{align*}
\partial_t u_\phi + \bu \cdot \nabla u_\phi + \frac{u_r u_\phi + u_\theta u_\phi \cot \theta}{r} - \frac{\LAP u_\phi}{\Reynolds} + 
\frac{\dphi p}{r \sin \theta} 
+ \frac{1}{\mathrm{Re}} \left( \frac{u_\phi}{r^2 \sin^2 \theta} - \frac{2 \cos \theta}{r^2 \sin^2 \theta}\partial_\phi u_\theta 
- \frac{2}{r^2 \sin \theta}\partial_\phi u_r \right) = 0 
 \end{align*}
 Let $\Lphir, \Lphitheta$ and $\Lphiphi$ be defined as follows:
$$\Lphir u  =  \left( \frac{\Drr}{\mathrm{Re}} - u^{*,n+1/2}_{2,r} \cdot \dr \right) u \text{ and }  
\Lphitheta u  =  \left(\frac{\DHthetatheta }{\Reynolds} - u^{*,n+1/2}_{2,\theta} \cdot \frac{\dtheta }{r} \right)u $$
$$ \Lphiphi u = \frac{1}{\Reynolds} \left(\DHphiphi - \frac{1}{r^2 \sin^2 \theta} \right)u - \frac{u_\phi^{*,n+1/2} \cdot u}{r \sin \theta}
- \frac{u_{2,r}^{*,n+1/2} + u^{*,n+1/2}_{2,\theta} \cot \theta}{r} u \text{ and } L_\phi = \Lphir + \Lphitheta + \Lphiphi $$
The factorized scheme for the $\phi$-component is then:
\begin{align}
& \left[ \mathrm{I} - \frac{\dt}{2} \Lphir \right]
\left[ \mathrm{I} - \frac{\dt}{2} \Lphitheta \right]
\left[ \mathrm{I} - \frac{\dt}{2} \Lphiphi \right]
\frac{ u_{1,\phi}^{n+1}-u_{1,\phi}^{n}}{\dt} = L_\phi u_{1,\phi}^{*,n+1/2} + \hat{\LAP} u^{n-1/2}_{1,\phi}
\notag \\
& - \frac{\dphi p_1^{n}}{r \sin \theta} + \frac{1}{\mathrm{Re}} \left(\frac{2}{r^2 \sin \theta} \dphi u_{1,r}^{n+1/2} + \frac{2 \cos \theta}{r^2 \sin^2 \theta} \dphi u_{1,\theta}^{n+1/2} \right) 
+ D_{31} u_{1,r}^{n+1/2} + D_{32} u_{1,\theta}^{n+1/2}. \label{eq:phi_comp}
\end{align}
\subsubsection{Pressure update}
\begin{equation}
p_1^{n+1} = p_1^{n} - \frac{1}{\chi}\DIV \bu_1^{n+1/2}. \label{eq:nst_Mass1}
\end{equation}

\section{Implementation and parallelization}
The equations \eqref{eq:Convect_Douglas}, \eqref{eq:r_comp}-\eqref{eq:phi_comp} are solved as a sequence of 1D equations in each space direction.
For example, solving \eqref{eq:Convect_Douglas} consists of the following steps:
\begin{align*}
 \frac{\xi^{n+1}}{\tau} & := \frac{1}{2} \Delta T^{*,n+1/2} - \frac{1}{2} \hat{\Delta} \delta T^n + \bu^{n+1/2} \cdot \nabla T^n \\
 \frac{\eta^{n+1}}{\tau} & := \left[ \mathrm{I} - \frac{\tau}{2} \DHthetatheta \right]
 \left[ \mathrm{I} - \frac{\tau}{2} \DHphiphi \right] \frac{T^{n+1}-T^n}{\tau} \Rightarrow 
 \left[ \mathrm{I} - \frac{\tau}{2} \Drr \right]\eta^{n+1} = \xi^{n+1} \\
 \frac{\zeta^{n+1}}{\tau} & := \left[ \mathrm{I} - \frac{\tau}{2} \DHphiphi \right] \frac{T^{n+1}-T^n}{\tau} \Rightarrow 
 \left[ \mathrm{I} - \frac{\tau}{2} \DHthetatheta \right]\zeta^{n+1} = \eta^{n+1} \\
 \left[ \mathrm{I} - \frac{\tau}{2} \DHphiphi \right] \left( T^{n+1} - T^n \right) & = \zeta^{n+1} \Rightarrow T^{n+1} = \left( T^{n+1}- T^n \right) + T^n.
\end{align*}
Similar strategy is applied for the Navier-Stokes approximation. 
Each 1D system is spatially approximated using second-order centered finite differences on a non-uniform grid. In order to ensure the inf-sup stability,
the unknowns are approximated on a MAC  grid, where the velocity components are stored at the face centers of the cells, while the scalar
variables are stored at the cell centers. 

To solve the system on each domain in parallel we use the approach developed in \cite{doi:10.1002/fld.2583}, where 
we first perform Cartesian domain decomposition of both computational grids using MPI, and then solve the resulting set of tridiagonal linear systems
using domain-decomposition-induced Schur complement technique. Note, that the Schur complement can be computed explicitly (see \cite{doi:10.1002/fld.2583} for details)
and so the system in each direction can be solved directly by the Thomas algorithm, avoiding the need of iterations on each of the two subdomain.
Then, in order to obtain the approximation on the entire spherical shell,  we iterate between the Yin and Yang grids using either additive or multiplicative 
overlapping Schwarz methods.  The solution on each grid is computed  using only boundary data  that is interpolated from the currently available 
solution on the other grid, using Lagrange interpolation.

In the additive Schwarz implementation, we use an even total number of CPUs. Then we split the global communicator
into two equal parts, and assign to each grid one of the communicators. In the multiplicative Schwarz implementation, we use the global 
communicator to solve the problem on each grid sequentially.  

The overall solution procedure in case of the multiplicative Schwarz iteration can be summarized as follows:
\begin{algorithm}
Repeat until convergence:	
\begin{itemize}
 \item[] For $i = 1,2$
    \begin{itemize}
      \item[1)] Obtain interpolated boundary values $T_{bd}$ for $\partial \Omega_i$ from $\Omega_{3-i}$.
      \item[2)] Solve the temperature equation in $\Omega_i$ with using extrapolated velocity values $\bu_2^{*,n+1/2}$.
      \item[3)] Obtain interpolated boundary values $\bu_{bd}$ for $\partial \Omega_i$ from $\Omega_{3-i}$.
      \item[4)] 
     If $\left| \bigintsss\limits_{\partial \Omega_i} \bu_{bd} \cdot \mathbf{n} \right| \ge \mathrm{tol}$, then minimize
      the functional ($\varepsilon \ll 1$):
      $$\mathrm{J}\left( \bv \right): = \frac{1}{2} \left|\bv - \bu_{bd} \right|^2_{\ell^2} 
      + \frac{1}{2 \varepsilon |\partial \Omega_i|^2} 
      \left| \bigintssss\limits_{\partial \Omega_i \cap \{\theta, \phi \text{ bdry }\}} \bv \cdot \mathbf{n} + 
      \bigintssss\limits_{\partial \Omega_i \cap \{r \text{ bdry }\}} \bu_{bd} \cdot \mathbf{n} \right|^2,$$
      using the Conjugate Gradient Algorithm until $\mathrm{J}\left( \cdot \right) \le \mathrm{tol}$.
      \item[5)] Update $\bu_{bd}: = \bv$ and solve the momentum equation in $\Omega_i$ with the interpolated Dirichlet boundary conditions in $\theta, \phi$ directions
      and with the original  boundary conditions in the $r$ direction. 
      \item[6)] Compute the pressure in $\Omega_i$ using the second equation in \eqref{eq:ac1}.
      \item[7)] Interpolate the pressure values at the boundary of $\partial \Omega_i$ using the available pressure on $\Omega_{3-i}$.
      \end{itemize}
\item[] End for.
\end{itemize}
\end{algorithm}

Step 4 is meant to ensure that there is no spurious mass flux generated 
through the internal (artificial) boundaries due to the interpolation. It is optional, and as our numerical experience shows, 
it rarely changes significantly the results.  Therefore, it is skipped while producing the numerical results presented in the next section.
 Skipping Step 7, however, can seriously reduce the rate of convergence of the Schwarz iteration, as observed in the numerical simulations. 
 Clearly, the AC method for the Navier-Stokes equations does not require boundary conditions on the pressure. Nevertheless, 
 the exchange of the pressure values does influence the pressure gradient that appears in \eqref{eq:r_comp}-\eqref{eq:phi_comp},
 and thus it seems to influence significantly the convergence of the overall iteration.
This effect is not well understood and while some other authors  (see for example \cite{TANG2003567}) also interpolate  the pressure values near the internal boundaries,
others (e.g. \cite{MERRILL201660} ) interpolate only the velocity on the internal boundaries.

Another interesting feature of the domain decomposition iteration described above  is that it allows   to use the
previously computed iterates in order to reduce the splitting error of the direction splitting approximation.  For example,
if the factorized form of the direction-split approximation for a given quantity $\psi$ is given by:
\[
(I-L_{\psi,r})(I-L_{\psi,\theta})(I-L_{\psi,\phi})(\psi^{n,k} - \psi^{n-1}) = G
\]
where the superscript $n$ denotes the time level of the solution and $k$ denotes the domain decomposition iteration level,
then the splitting error can be reduced by using the modified equation:
\begin{equation}
(I-L_{\psi,r})(I-L_{\psi,\theta})(I-L_{\psi,\phi})(\psi^{n,k} - \psi^{n,k-1}) = G+(\psi^{n,k-1}-\psi^{n-1})-L_{\psi} (\psi^{n,k-1}-\psi^{n-1}),
\label{eq:er}
\end{equation}
where $L_{\psi}= L_{\psi,r}+L_{\psi,\theta}+L_{\psi,\phi}$.
Indeed, in \eqref{eq:er}, the splitting error term $(L_{\psi,r}L_{\psi,\theta}+L_{\psi,r}L_{\psi,\phi}+L_{\psi,\theta}L_{\psi,\phi}-L_{\psi,r}L_{\psi,\theta}L_{\psi,\phi})(\psi^{n,k} - \psi^{n-1}) $
at iteration level $k$ has been reduced by the same term at the previous iteration level
$(L_{\psi,r}L_{\psi,\theta}+L_{\psi,r}L_{\psi,\phi}+L_{\psi,\theta}L_{\psi,\phi}-L_{\psi,r}L_{\psi,\theta}L_{\psi,\phi})(\psi^{n,k-1} - \psi^{n-1}) $.
If this error reduction is employed, then the iteration becomes a block-preconditioned overlapping domain decomposition iteration, 
the preconditioner being the factorized operator $(I-L_{\psi,r})(I-L_{\psi,\theta})(I-L_{\psi,\phi})$. We must also remark here that
this iteration needs to converge to an accuracy of the order of $\dt^2$ for the solution of equation \eqref{eq:ac1}
and $\dt^3$ for the solution of equation \eqref{eq:ac2}, in order to preserve the second order accuracy of the overall
algorithm.
\section{Numerical tests}
\subsection{Time and space convergence}
We verify the convergence rates in space and time using the following manufactured solution, given in a Cartesian form:
\begin{equation}
\bu = \cos(t) \left(2 x^2 y z, - x y^2 z, -x y z^2 \right)^\mathrm{T} , p = \cos(t) x y z, T = 2 \cos(t) x^2 y z.
\label{eq:Exact}
\end{equation}
The parameters used in this test  are $R_1=1, R_2=2, Ra=1, Pr=1$, and the grids used in the tests are uniform in each direction.
The convergence of the approximation is tested using both, the additive and multiplicative versions of the scheme. 
The grid used for the time convergence tests consists of $20\times 92 \times 192 $ MAC cells on each of the two
subdomains.  The solution error is computed at the final time $T_f=10$.
For the space convergence tests, the time step is chosen small enough to not influence the overall error, $\dt=0.0001$, and the final time is $T_f=1$.
The grid diameter is computed as the maximum diameter of the MAC cells in Cartesian coordinates.
In both cases the domain decomposition iterations are converged so that the $l^2$ norm of the difference between two subsequent
iterates, for any of the computed quantities is less than $10^{-6}$ ($l^2$ norm denotes the standard mid-point approximation to the $L^2$ norm).
Also, the splitting error reduction, as outlined by equation \eqref{eq:er}, is employed at each iteration.

The graphs of the $l^2$ norm of the errors in both cases are  presented in Figure \ref{fig:Conv_Additive_Schwarz}.
They clearly demonstrate  the second order accuracy of the scheme in space and time.

Next we verify the accuracy of the proposed algorithm on a physically more relevant analytic solution of the Navier-Stokes equations
in a spherical setting, due to Landau (see \cite{L44} and \cite{LYY18} for a recent review).  The source term of the equations
is equal to zero in this case, and the solution is steady and axisymmetric.  In all cases presented in figure \ref{fig:Conv_Landau} 
the multiplicative Schwarz version of the algorithm is used with its convergence tolerance being set to $10^{-6}$, the time step
is equal to $10^{-3}$, and $R_1=1, R_2=2$.
We first present in the top left graph of figure \ref{fig:Conv_Landau} the $l^2$ error for the velocity and pressure as a function of the grid diameter,
at $Re=1$ and the overlap is $\epsilon=0.1$
The scheme clearly exhibits again a second order convergence rate in space.  In the top right graph we demonstrate 
the influence of the overlap size on the error at $Re=1$, the grid size in the $r,\phi, \theta$ directions being 
$2.7778 \times 10^{-2},    1.7027 \times 10^{-2},    3.6121 \times 10^{-2}$ correspondingly.  
The effect of the overlap on the error is insignificant, however, it seriously impacts
the stability of the algorithm i.e. the increase of the overlap improves the stability, particularly at large Reynolds numbers.
Finally, the bottom graph demonstrates the effect of the Reynolds number  on the error.  Again, the overlap is $0.1$ and the grid sizes
are equal to $2.7778 \times 10^{-2},    1.7027 \times 10^{-2},    3.6121 \times 10^{-2}$ . We should note that the exact solution for the velocity scales
like $Re^{-1}$ and therefore the errors in the graph are multiplied by the corresponding Reynolds number.  Clearly, the oscillations in the error
decrease slower with the increase of the Reynolds number.  These oscillations are due to the artificial compressibility algorithm,
since the initial data for the pressure corresponds to a divergence-free velocity, while the pressure evolution is determined by
a perturbed continuity equation (see \cite{OA10} and \cite{DLM17} for a detailed discussion on this issue).

\begin{figure}[htbp]
\centering
\subfigure{\includegraphics[width=80mm]{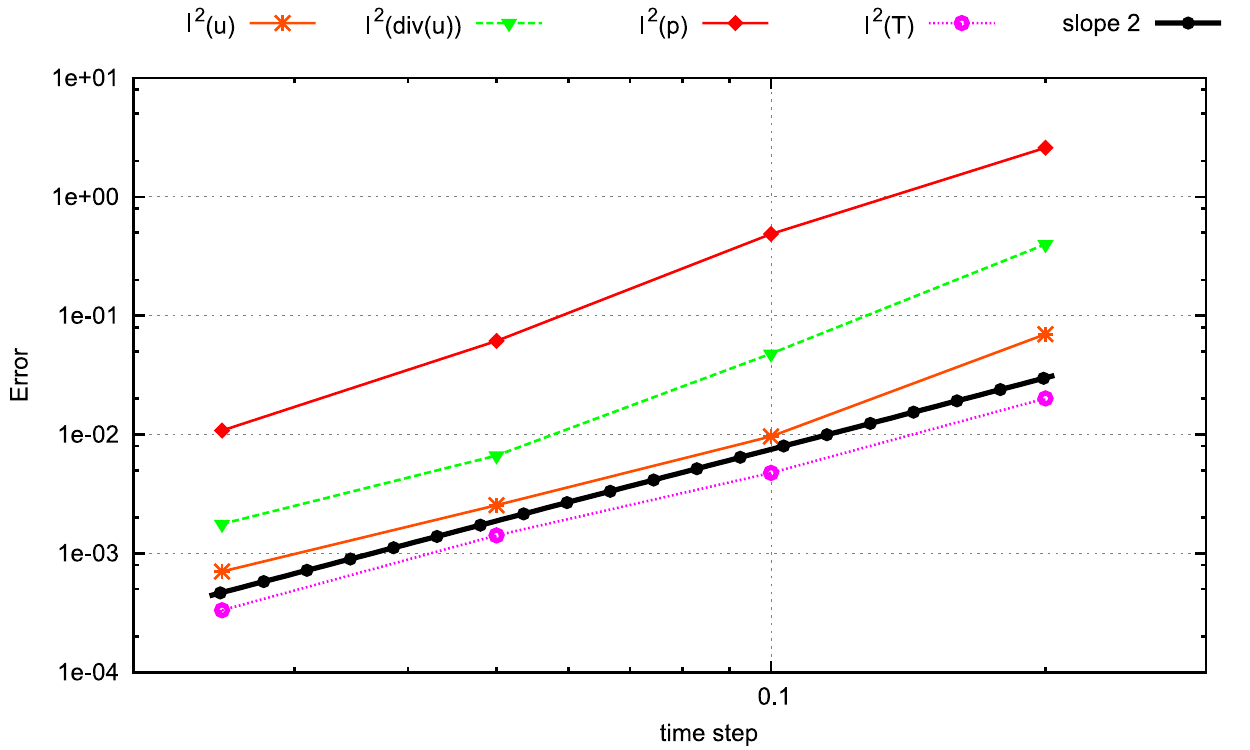}}
\hspace{0.01cm}
\subfigure{\includegraphics[width=80mm]{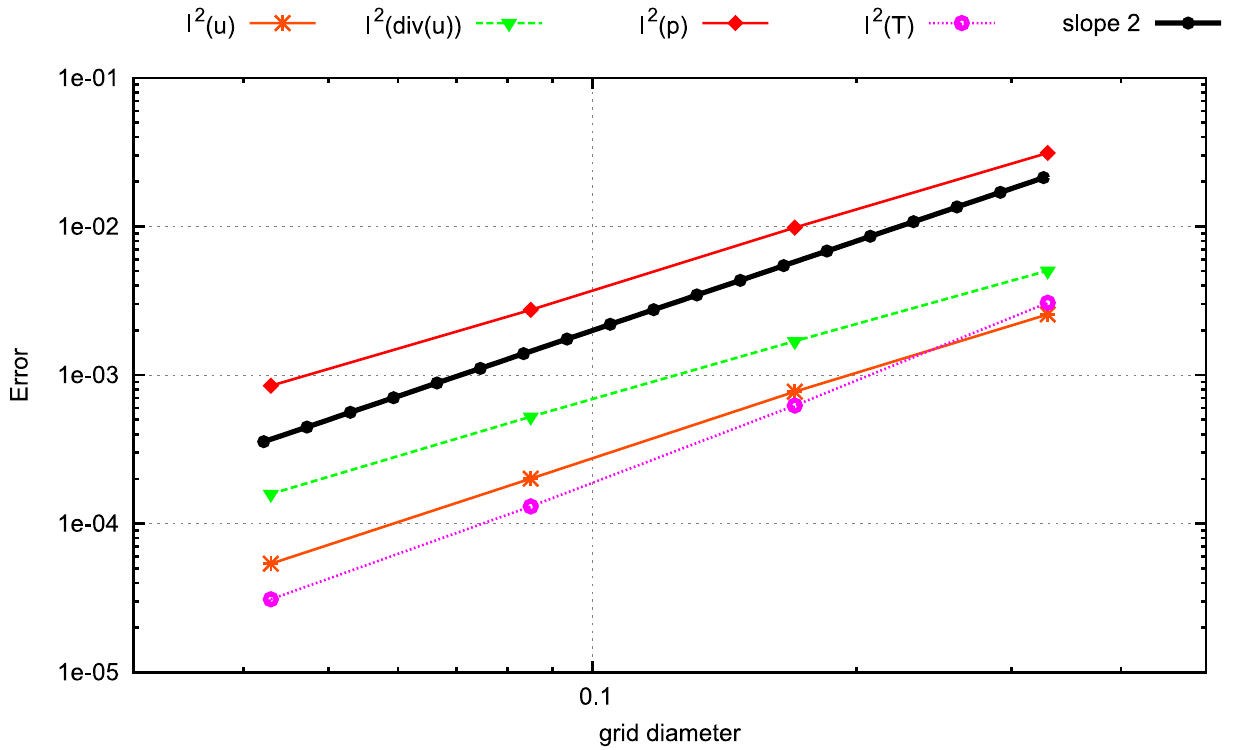}}
\caption{Log-log plot of the errors; multiplicative Schwarz iteration. Left graph contains the temporal errors at $T_f=10$, 
while the right graph contains the spatial error plotted at $T_f=1$; $R_1=1, R_2=2, Ra=1, Pr=1$.}
\label{fig:Conv_Additive_Schwarz}
\end{figure}
\begin{figure}[htbp]
\centering
\subfigure{\includegraphics[width=80mm]{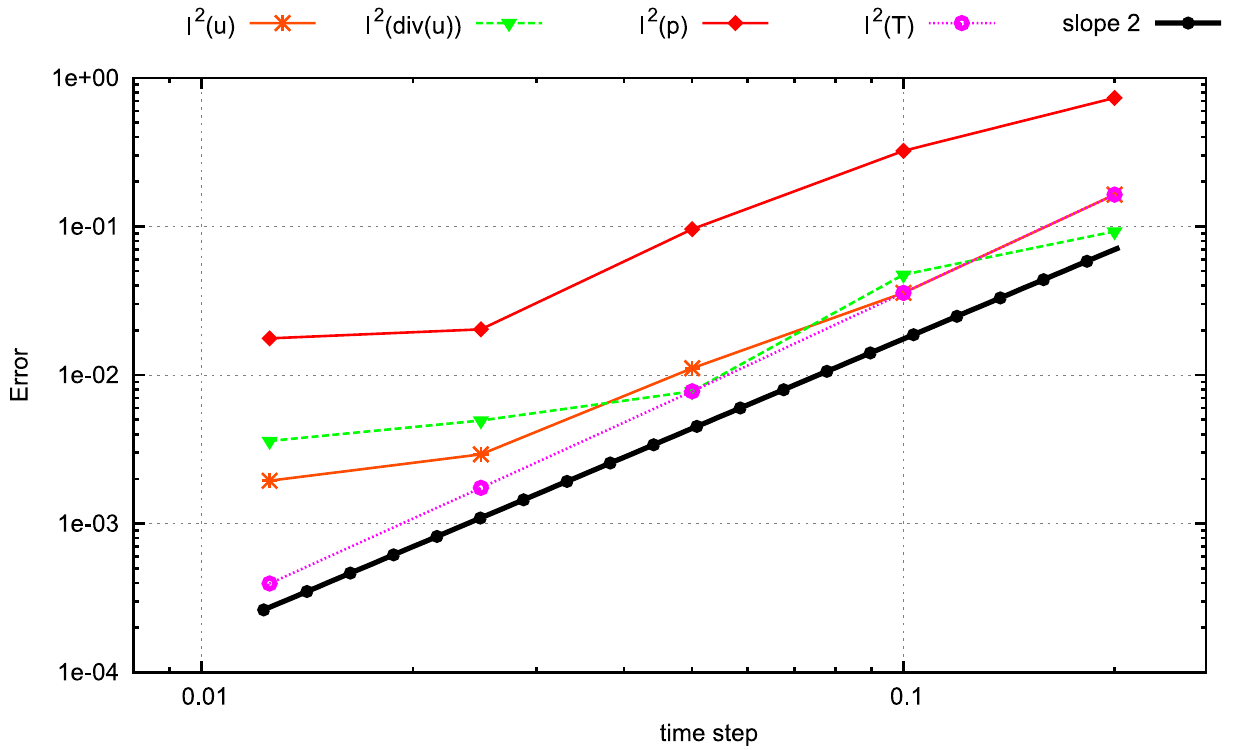}}
\hspace{0.01cm}
\subfigure{\includegraphics[width=80mm]{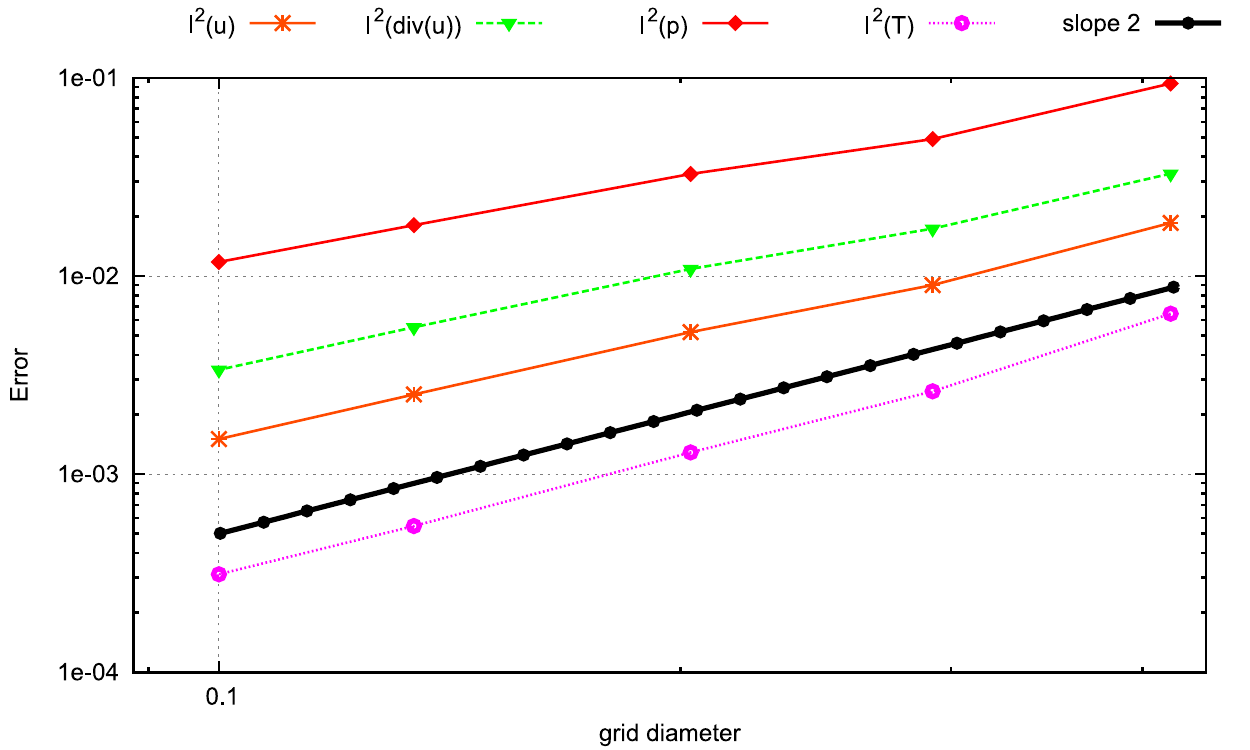}}
\caption{Log-log plot of the errors; additive Schwarz approach. Left graph contains the temporal errors at $T_f=10$, 
while the right graph contains the spatial error plotted at $T_f=1$; $R_1=1, R_2=2, Ra=1, Pr=1$.}
\label{fig:Conv_Multiplicative_Schwarz}
\end{figure}

\begin{figure}[htbp]
\centering
\subfigure{\includegraphics[width=80mm]{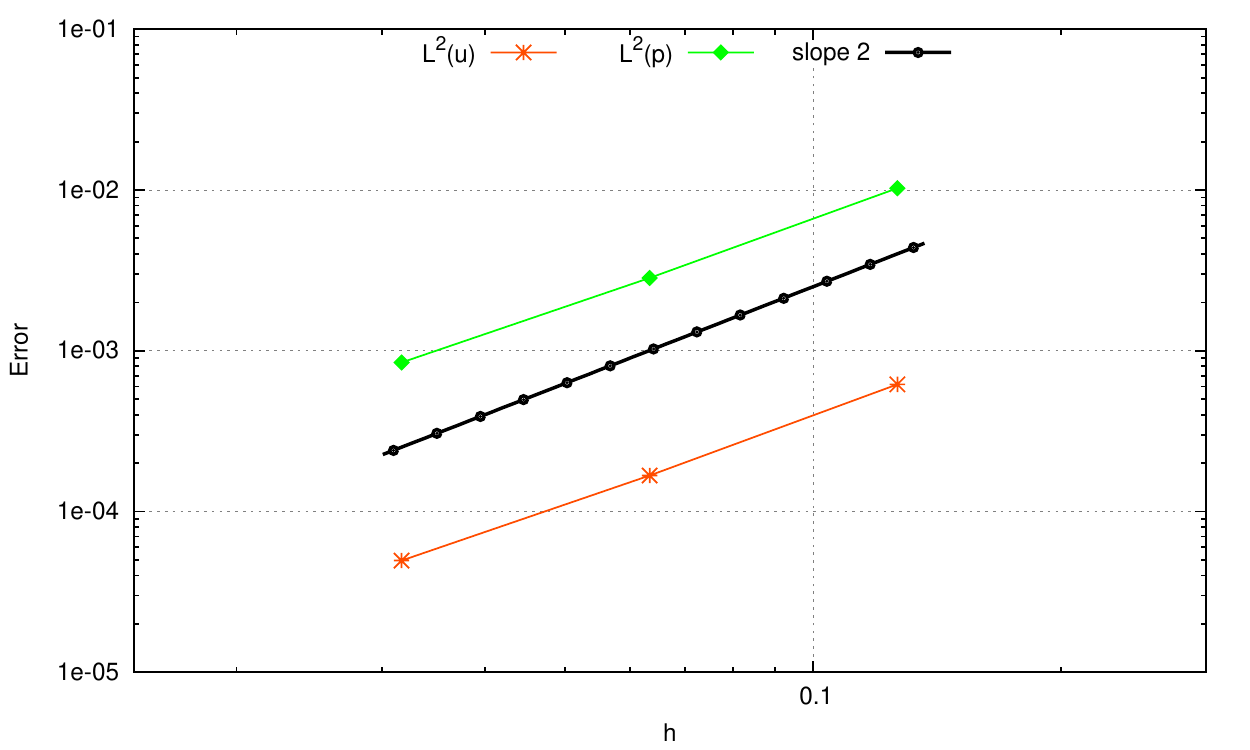}}
\hspace{0.01cm}
\subfigure{\includegraphics[width=80mm]{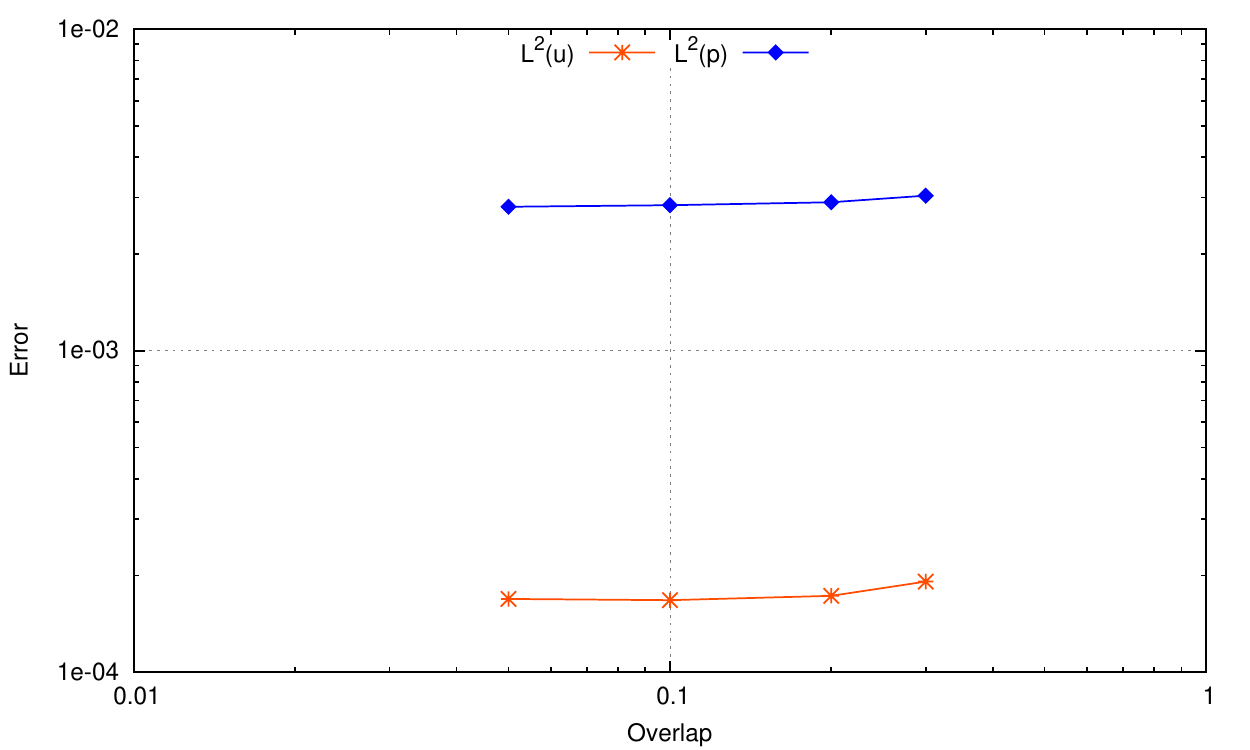}}
\\
\subfigure{\includegraphics[width=80mm]{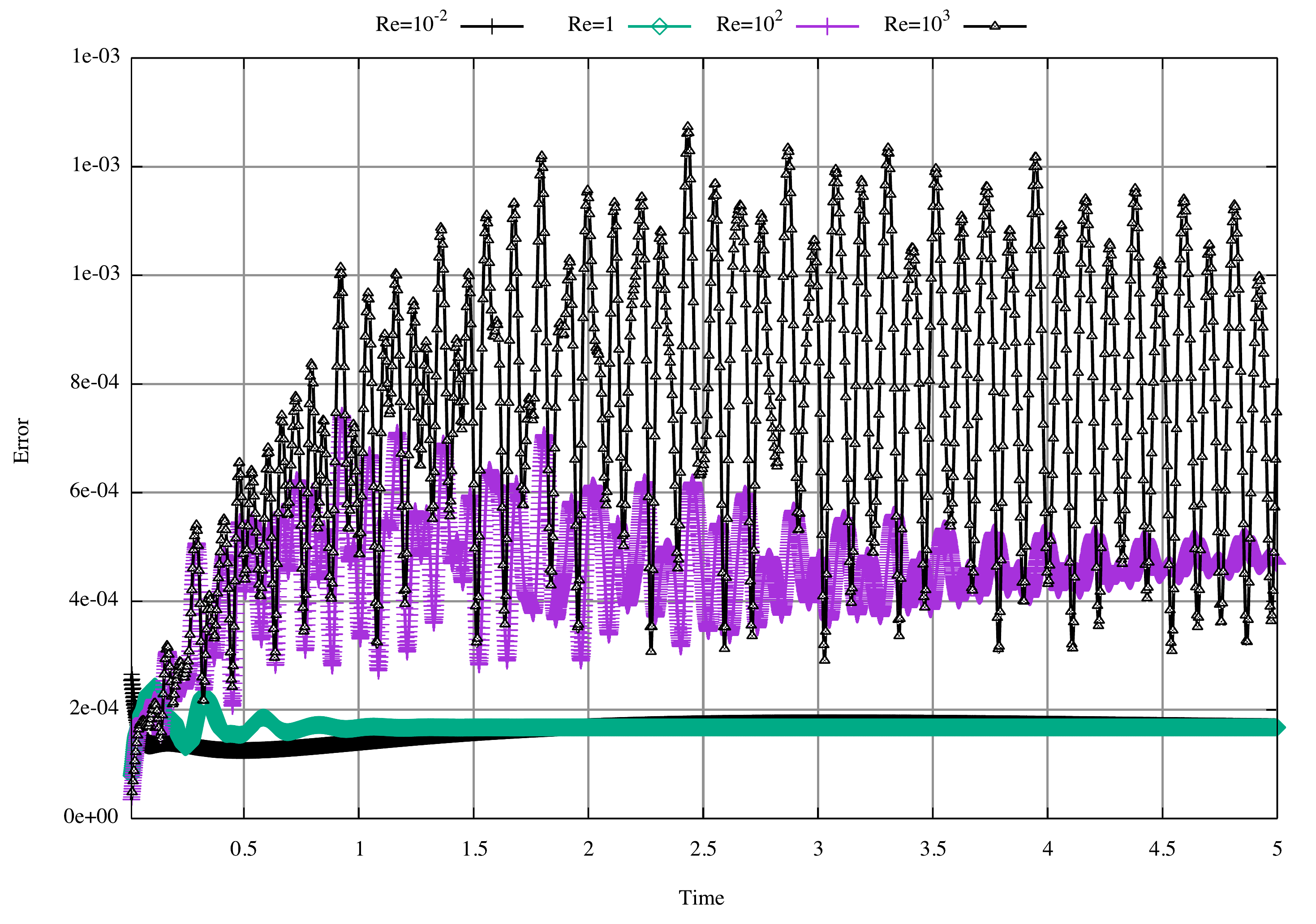}}
\caption{$l^2$ errors. Top left: convergence in space, Re=1. Top right: effect of the overlap on the error; Re=1.
Bottom: effect of the Reynolds number on the velocity error as a function of time. }
\label{fig:Conv_Landau}
\end{figure}
\subsection{Weak parallel efficiency}
Next we test the parallel efficiency of the code based on the scheme introduced in the previous section. Since we are interested in solving large size problems, we 
only measure the weak scalability of our code. The problem size is $100\times 100 \times 100$ grid cells per each of the Yin and Yang  grids on each CPU,
and the maximum number of CPUs used is $960$.
Besides, since in the possible applications of this technique (atmospheric boundary layer, Earth's dynamo)
 the thickness of the spherical shell is much smaller than the diameter of the shell, we use a two-dimensional grid of processors for the grid partitioning.
 It must be noted though, that making the grid partitioning three dimensional does not change much the parallel efficiency results presented in this section.
 The scaling efficiency is computed as the ratio of the CPU time on 32 cores divided by the CPU time on $n\geq32$ cores. The reason for this definition
 of efficiency is that the particular cluster used in the scaling tests has processors containing 32 cores each, and the efficiency drops very significantly between
 1 and 32 cores (to about 75\%). After this, when the number of cores is a multiple of 32 the efficiency remains very close to the one at 32 cores.  One possible
 explanation of this phenomenon is that in case when the number of cores is significantly less than 32 cores, they need to share the memory bandwidth and cache with a smaller 
 number of cores, since presumably the rest of the available cores on the given processor are idle (see e.g. \cite{Keating}, p. 152).   Again, we are interested in very large computations, and therefore,
 using a minimum of  32 cores is very reasonable.
 
 The scaling results are performed using the Compute Canada (see https://www.computecanada.ca/) Graham cluster of 
$2.1$GHz Intel $E5-2683$ v4 CPU cores, $32$ cores per node, and each node connected via a $100$ Gb/s network. 
The results were calculated using the wall clock time taken to simulate $10$ time steps.
 We ran two tests, using a fixed number of $1$ and $10$ domain decomposition iterations, and we present the scaling results in  Fig. \ref{fig:Scaling}.
 The parallel efficiency is very slightly dependent on the number of domain decomposition iterations, and remains above 90\% for the number of cores ranging 
 between 32 and 960 (the maximum allocatable without a special permission on the particular cluster).  In our opinion, this is an excellent scaling result
 for an implicit scheme for the incompressible Boussinesq equations.
%
\begin{figure}[htbp]
\centering
\includegraphics[width=100mm]{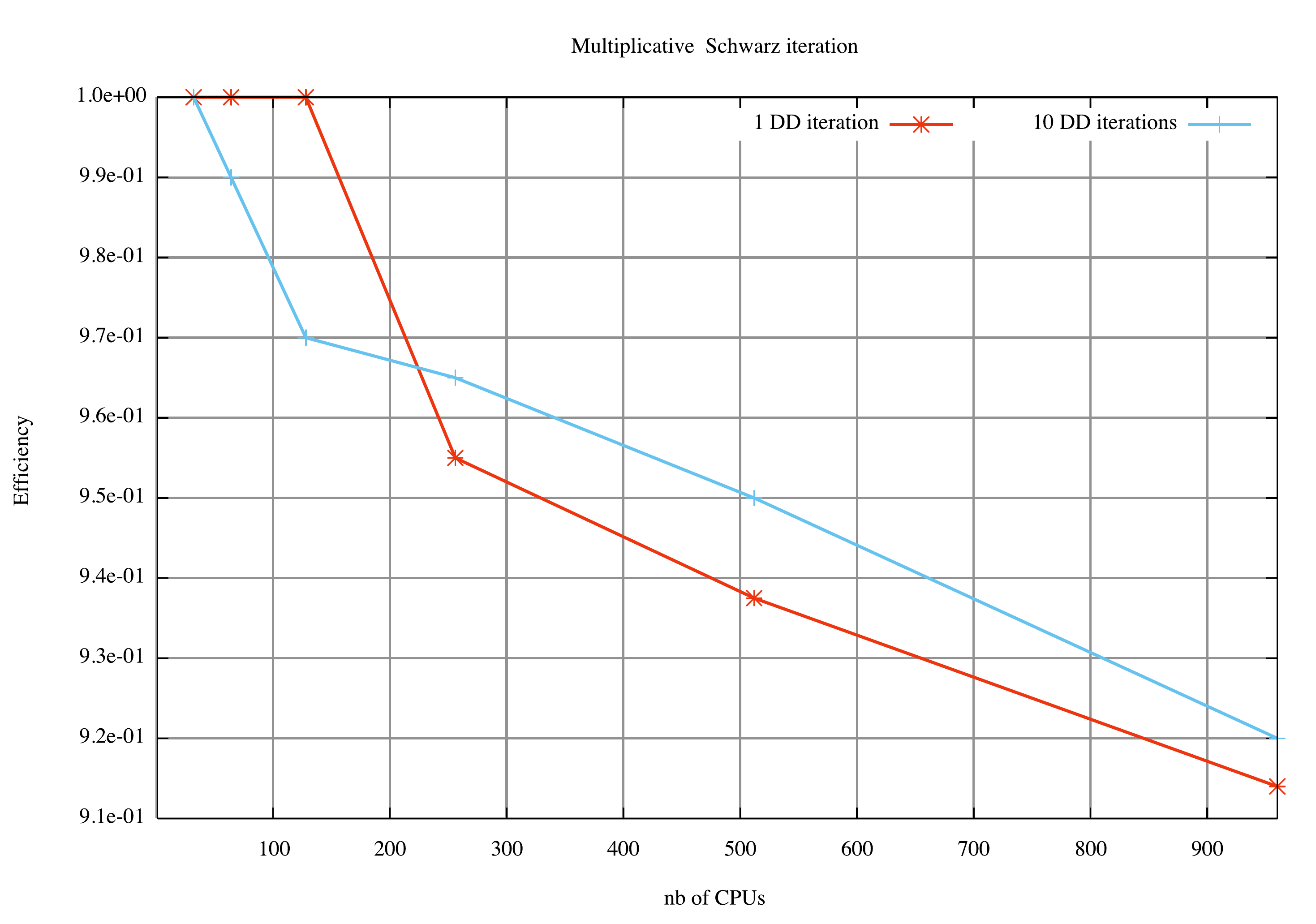}
\caption{Parallel scalability using up to $960$ CPU cores}
\label{fig:Scaling}
\end{figure}
%
\section*{Acknowledgments}
The authors would like to acknowledge the support, under a Discovery Grant,
of the National Science and Engineering Research Council of Canada (NSERC).

This research was enabled in part by support provided by Compute Canada (www.computecanada.ca).

\section*{References}

\bibliography{references}

\end{document}